\documentclass[11pt]{amsart}
\usepackage[top=1in, bottom=1in, left=1in, right=1in]{geometry}
\expandafter\let\csname ver@amsthm.sty\endcsname\relax

\usepackage{graphicx} 

\usepackage{tikz}

\usepackage{amsmath}
\usepackage{amssymb}
\usepackage{mathdots}
\usepackage{mathtools}

\usepackage{dsfont}
\usepackage{chngpage}

\usepackage{hyperref}
\usepackage{amsthm}
\usepackage[capitalize,noabbrev]{cleveref}
\usepackage{subcaption}

\allowdisplaybreaks

\numberwithin{equation}{section}

\newtheorem{thm}{Theorem}[section]
\newtheorem{lemma}[thm]{Lemma}
\newtheorem{cor}[thm]{Corollary}
\newtheorem{prop}[thm]{Proposition}

\newtheorem{Example}[thm]{Example}

\newtheorem{Remark}[thm]{Remark}

\crefname{thm}{Theorem}{Theorems}
\crefname{lemma}{Lemma}{Lemmas}
\crefname{cor}{Corollary}{Corollaries}
\crefname{prop}{Proposition}{Propositions}

\crefname{example}{Example}{Examples}
\crefname{remark}{Remark}{Remarks}

\newcommand{\emailhref}[1]{\email{\href{#1}{#1}}}

\newcommand{\Des}{\mathrm{Des}}
\newcommand{\des}[2]{\mathrm{d}(#1;#2)} 
\newcommand{\size}[1]{|#1|} 
\newcommand{\asc}[1]{T^{\uparrow #1}} 
\newcommand{\desc}[1]{T^{\downarrow #1}} 
\newcommand{\order}[0]{s} 
\newcommand{\degree}[0]{h} 
\newcommand{\abar}[1]{\bar{a}_{#1}(T)} 
\newcommand{\descset}[0]{V_D} 
\newcommand{\ascset}[0]{V_A} 
\newcommand{\Natlab}[1]{N_{#1}} 
\newcommand{\natlab}[1]{n_{#1}} 

\title[Descent polynomials for labeled trees]{Descent polynomials for labeled trees}

\author[S. Poznanovi\'c]{Svetlana Poznanovi\'c}\emailhref{spoznan@clemson.edu}
\address{School of Mathematical and Statistical Sciences, Clemson University, Clemson, SC 29634}
\thanks{S.P. was supported by NSF DMS 1815832.}

\author[M. Rodriguez Hertz]{Maria Rodriguez Hertz}\emailhref{mdr1@williams.edu}\address{Department of Mathematics and Statistics, Williams College, Williamstown, MA 01267}

\author[S. Valore-Caplan]{Solomon Valore-Caplan}\emailhref{svalorecaplan@g.hmc.edu}\address{Department of Mathematics, Harvey Mudd College, Claremont, CA 91711}

\author[D. Wichmann]{David Wichmann}\emailhref{david.wichmann@go.mnstate.edu}\address{Department of Computer Science, Minnesota State University Moorhead, Moorhead, MN 56563}



\begin{document}

\begin{abstract} Motivated by the properties of the descent polynomials, which enumerate permutations of $S_n$ with a fixed descent set, we define descent polynomials for labeled rooted trees. We give recursive and explicit formulas for these polynomials and show when known properties of the descent polynomials carry over to the setting of trees.
\end{abstract}

\maketitle


\section{Introduction} \label{S:introduction}

Let $T$ be a rooted tree with a vertex set $V$ of size $\order$. We will draw rooted trees with the root on top. For $v \in V$, let $p(v)$ denote the parent of $v$. A labeling of $T$ is a bijection $w: V \rightarrow [\order]$. The descent set of a labeling $w$ is \[\Des(w) = \{v \in V \colon w(v) > w(p(v))\}.\] In particular, the root is never included in the descent set. The set of natural labelings of $T$ is \[ \Natlab{T}= \{w \colon w \text { is a labeling of } T,\Des(w) = \emptyset\}. \] It is well-known that the number of natural labelings of $T$ is given by the following hook-length formula
\[\natlab{T} = \frac{n!}{\prod_{v \in V(T)}h_v},\] where $h_v$ is the size of the subtree $T_v$ of $T$ rooted at the vertex $v$.

For an integer $n \geq \order$, let $G(T, n)$ be the tree with $n$ vertices obtained by adding a chain of size $n-\order$ above the root of $T$. Let $\descset \subseteq V$.  Let  \[ D(T;n) = \{ w \colon w \text { is a labeling of } G(T;n),\mathrm{Des}(w) = \descset \}\] and 

\[\des{T}{n} = |D(T; n)|.\] Note that $D(T;n)$ and $\des{T}{n}$ depend on $\descset$ but, since in what follows $\descset$ will be fixed, we keep the notation simpler by not making this dependence explicit. One can readily see that $\des{T}{\order} =0$ if and only if $\descset$ contains the root of $T$. Moreover, $\des{T}{n} > 0$ for $n > \order$. This can be seen by the following construction of a labeling in $D(T;n)$: label the vertices of $G(T;n)$ in $\descset$ by the numbers $n$, $n -1$, \dots, starting from the lowest generation and moving up, and then   label the remaining vertices by $1$, $2$, \dots, again starting at the lowest generation and moving up.

From now on, when we refer to a tree $T$, we mean $T$ with the distinguished set of vertices $\descset$.  The vertices in $\descset$ are called \emph{descent vertices} and will be represented in the figures as black nodes. All other vertices in $G(T;n)$ are called \emph{ascent vertices} and are colored white. See Figure~\ref{fig:Extending-to-n-vertices} for an example. 

When $T$ is a chain, the labelings in $D(T;n)$ correspond to permutations of $n$ with a fixed descent set. The function $\des{T}{n}$ was shown to be a polynomial in $n$ by MacMahan~\cite{macmahon2001combinatory} in 1915. However, there appears to not have been any study of this polynomial until recently, when Diaz-Lopez, Harris, Insko, Omar, Sagan~\cite{diaz2019descent} initiated a study of $\des{T}{n}$, motivated by the properties of the analogous peak polynomial for permutations~\cite{billey2013permutations,diaz2017proof}. After that, q-analogues of the peak and descent polynomials were studied in~\cite{gaetz2021q}. A generalziation of the descent polynomials to permutations of multisets was studied in~\cite{raychev2023generalization}.

Labeled trees are a natural extension of permutations. A lot of classical permutations statistics, including descents, have analogues in labeled trees and share many of the  properties~\cite{grady2020tree, gonzalez2016note}. We will show that some of the established properties of the descent polynomials for permutations also hold in the tree setting. First, in Section~\ref{S:computing}, we prove that $\des{T}{n}$ is a polynomials and we find its degree. We also give recursions   as well an explicit formula for computing $\des{T}{n}$. Then, in Section~\ref{S:expansions}, we consider expansions of $\des{T}{n}$ in certain binomial bases and prove results about the coefficients in those expansions. Finally, Section~\ref{sec:roots} is devoted to understanding the roots of $\des{T}{n}$.

    \begin{figure}[ht]
        \centering
        \begin{subfigure}[b]{0.3\textwidth}
            \centering
            \begin{tikzpicture}[level distance = 1cm, sibling distance=1cm, nodes={draw, circle, minimum size = 6mm}, -,
    		        ascent/.style={circle, draw=black, fill=white, text=black, thick, radius=0.3},
    		        descent/.style={circle, draw=black, fill=black, text=white, radius=0.3},
    		        path/.style={circle, draw=white, fill=white, text=black, radius=0.3}
    		        ]
                \node[ascent, label=left:$v_6$] {}

                                child { node [descent, label=left:$v_4$]  {} 
                                    child { node [ascent, label=left:$v_1$] {} }
                                    child { node [descent, label=right:$v_2$] {} }}
                                child [missing]
                                child { node [ascent, label=right:$v_5$]  {} 
                                    child { node [ascent, label=right:$v_3$] {} }
                    };
            \end{tikzpicture}
            \caption{A tree $T$}
        \end{subfigure}
        \begin{subfigure}[b]{0.3\textwidth}
            \centering
            \begin{tikzpicture}[level distance = 1cm, sibling distance=1cm, nodes={draw, circle, minimum size = 6mm}, -,
		        ascent/.style={circle, draw=black, fill=white, text=black, thick, radius=0.3},
		        descent/.style={circle, draw=black, fill=black, text=white, radius=0.3},
		        path/.style={circle, draw=white, fill=white, text=black, radius=0.3}
		        ]
                \node[ascent, label=left:$v_n$] {}
                    child{ node [path] {$\vdots$}
                        child{ node [ascent, label=left:$v_6$] {}
                            child { node [descent, label=left:$v_4$]  {} 
                                child { node [ascent, label=left:$v_1$] {} }
                                child { node [descent, label=right:$v_2$] {} }}
                            child [missing]
                            child { node [ascent, label=right:$v_5$]  {} 
                                child { node [ascent, label=right:$v_3$] {} }}}
                };
        \end{tikzpicture}
        \caption{The tree $G(T, n)$}
        \end{subfigure}
        \begin{subfigure}[b]{0.3\textwidth}
            \centering
            \begin{tikzpicture}[level distance = 1cm, sibling distance=1cm, nodes={draw, circle, minimum size = 6mm}, -,
    		        ascent/.style={circle, draw=black, fill=white, text=black, thick, radius=0.3},
    		        descent/.style={circle, draw=black, fill=black, text=white, radius=0.3},
    		        path/.style={circle, draw=white, fill=white, text=black, radius=0.3}
    		        ]
                \node[ascent] {7}
                        child { node [ascent] {6} 
                            child { node [ascent] {4}
                                child { node [descent]  {5} 
                                    child { node [ascent] {2} }
                                    child { node [descent] {8} }}
                                child [missing]
                                child { node [ascent]  {3} 
                                    child { node [ascent] {1} }}}
                    };
            \end{tikzpicture}
            \caption{A labeling in $D(T; 8)$}
        \end{subfigure}
        \caption{}
        \label{fig:Extending-to-n-vertices}
    \end{figure}

\section{Computing $\des{T}{n}$} \label{S:computing}

Let $\ascset = V \setminus \descset$. For $v \in V$, let $T_v$ denote the subtree of $T$ rooted at $v$. Let $h_v$ be the number of vertices in $T_v$; this is also known as the \emph{hook length} of $v$. For $v \in \descset$, let $\asc{v}$  denote $T$ with $v$ changed from a descent to an ascent. For $v \in \ascset$, let  $\desc{v}$ be $T$ with $v$ changed from an ascent to a descent. We use $T \setminus T_v$ to denote the tree obtained from $T$ by deleting the subtree $T_v$, with distinguished descent vertices inherited from $T$. For a tree $T'$ derived from $T$ with one of these operations, we will generally use $\descset(T')$ to denote the set of descent vertices of $T'$ inherited from $\descset$. We first give a recursive formula for computing $\des{T}{n}$.

\begin{prop} \label{prop:recursion}
Let $T$ be a rooted tree of size $\order$ with a distinguished subset of vertices $\descset$ and let $n \geq \order$. For $v \in \descset$, we have
\begin{equation} \label{rec}
    \des{T}{n} = \binom{n}{h_v} \cdot \des{\asc{v}_v}{h_v} \cdot \des{T \setminus T_v}{n-h_v} - \des{\asc{v}}{n},
\end{equation}
where, by convention,  $\des{\emptyset}{k} = 1$ for $k \geq 0$.
\end{prop}

\begin{proof}
Consider the set 
\[P = \{w \colon w \text{ is a labeling of } G(T, n) , \mathrm{Des}(w) = \descset \text{ or } \mathrm{Des}(w) = \descset \setminus \{v\} \}. \] 
Then, clearly, $|P| = \des{T}{n} + \des{\asc{v}}{n}$. On the other hand, a labeling $w \in P$ can be obtained uniquely by:  
$(1)$ labeling the subtree $T_v$ by $h_v$ of the $n$ available labels so that the vertices in $(\descset \setminus \{v\}) \cap T_v$ are descents 
and 
$(2)$ labeling the remaining part of the tree $T \setminus T_v$ with the remaining $n-h_v$ labels so that the vertices in  $\descset \cap (T \setminus T_v)$ are descents.
This yields  $|P| = \binom{n}{h_v} \cdot \des{\asc{v}_v}{h_v} \cdot \des{T \setminus T_v}{n-h_v}$ and the recurrence follows.
\end{proof}

A descent vertex $v \in \descset$ is \emph{maximal} if none of its ancestors is in $\descset$. 

\begin{cor} \label{cor:degree}
Let $T$ be a rooted tree, let $\{v_1, \dots, v_m \}$ be the set of its maximal descent vertices, and let $\degree = h_{v_1} + h_{v_2} + \dots + h_{v_m}$. Then $\des{T}{n}$ is a polynomial of degree $\degree$.
\end{cor}

\begin{proof}

We proceed by induction on $\degree$. First we check the case $\degree = 0$, i.e., when $T$ has no descents. Note that in this case, the top chain of $n - \order$ vertices in $G(T, n)$ must be labeled from top to bottom by the labels $n, n -1, \dots, n - \order$, so $d(T;n) = \natlab{T}$, a constant. 

Now, let $\order \geq 1$. Let $v$ be one of the maximal vertices of $T$. Consider the terms in~\eqref{rec}. Note the following:
\begin{itemize}
    \item $\binom{n}{h_v}$ is a polynomial of degree $h_v$.
    \item $\des{\asc{v}_v}{h_v}$ is a constant since it is not dependent on $n$.
    \item The maximal descents of $T \setminus T_v$ are contained in the set of maximal descents of $T$. So, by the inductive hypothesis, $\des{T \setminus T_v}{n}$, and therefore, $\des{T \setminus T_v}{n-h_v}$ as well, are a polynomials of degree $h - h_v$. 
    \item The maximal descent vertices of $\asc{v}$ that are not maximal descents in $T$ are in the subtree $T_v$, so by the inductive hypothesis,  $\des{\asc{v}}{n}$ is a polynomial of degree $< \degree$.
\end{itemize}

Thus, by Proposition~\ref{prop:recursion},
\begin{align*}
    \des{T}{n} &= (\text{poly. of degree }h_v)\cdot (\text{constant}) \cdot (\text{poly. of degree }\degree - h_v) 
 - (\text{poly. of degree }< \degree)  \\
    &= (\text{poly. of degree } \degree)
\end{align*}
as desired.
\end{proof}

For this reason, we also refer to $\des{T}{n}$ as \emph{the descent polynomial} of $T$. Note that instead of applying a recursion at descent vertices as in Proposition~\ref{prop:recursion}, one can do it at the ascent vertices, as illustrated by the following corollary. 

\begin{cor}\label{cor:ascent recursion}
Let $T$ be a rooted tree of size $\order$ with a distinguished set of vertices $\descset$ and let $n \geq \order$. Let $v \in \ascset$. Then
\begin{equation}\label{eq:recas}
    \des{T}{n} = \binom{n}{h_v} \cdot \des{\asc{v}_v}{h_v} \cdot \des{T \setminus T_v}{n-h_v} - \des{\desc{v}}{n}
\end{equation}
\end{cor}
\begin{proof}
If we apply Proposition~\ref{prop:recursion} to the tree $\desc{v}$ and the vertex $v$, we get an equality which is equivalent to~\eqref{eq:recas}.

\end{proof}

Our next result gives an explicit formula for $\des{T}{n}$ but we introduce some notation first. Let $N \subseteq \descset$. Suppose we delete from $T$ the edges between the vertices in $N$ and their parents. This procedure yields $|N|+1$ smaller trees rooted at the vertices of $N$ and the root of $T$: $T_0, \dots, T_{|N|}$, with descent sets inherited from $\descset$. Let $V(T_i)$ denote the set of vertices of $T_i$, $0 \leq i \leq |N|$.

\begin{prop}\label{prop:explicit}
Let $T$ be a rooted tree with descent set $\descset$, then
\begin{equation}\label{eq:fla}
    \des{T}{n} = n! \cdot \sum_{N \subseteq \descset} (-1)^{\size{\descset} - \size{N}} \cdot  \prod_{i = 0}^{|N|} \frac{1}{\prod_{v \in V(T_i)} h_v},
\end{equation}
where the hook length $h_v$ of the vertex $v \in V(T_i)$ is calculated within the tree $T_i$.
\end{prop}

\begin{proof}
We proceed by induction on $\size{\descset}$. If $\descset = \emptyset$, then $\des{T}{n} = \natlab{T}$, the right-hand side of ~\eqref{eq:fla} has only one term and we get the well-known hook length formula for $\natlab{T}$.   Assume that the statement holds for $\size{\descset} \leq k$.  

Let $\size{\descset} = k + 1$ and $v \in \descset$. Using Proposition~\ref{prop:recursion} and the inductive hypothesis, we get

\begin{align} 
    \des{T}{n} =& h_v! \cdot \binom{n}{h_v} \sum_{L \subseteq \descset(\asc{v}_v)} (-1)^{\size{\descset(\asc{v}_v)} - \size{L}} \cdot  \prod_{i=0}^{\size{L}} \frac{1}{\prod_{u \in V(T_i)} h_u} \nonumber\\ \label{eq:exp}
    &\cdot (n - h_v)! \cdot \sum_{J \subseteq \descset(T \smallsetminus T_v)} (-1)^{\size{\descset(T \smallsetminus T_v)} - \size{J}} \cdot  \prod_{i=0}^{\size{J}} \frac{1}{\prod_{u \in V(T_i)} h_u}\\
    &- n! \cdot \sum_{R \subseteq \descset(\asc{v})} (-1)^{\size{\descset(\asc{v})} - \size{R}} \cdot  \prod_{i=0}^{\size{R}} \frac{1}{\prod_{u \in V(T_i)} h_u} \nonumber
\end{align}

There is clearly a common factor in front of all terms, $n!$. Multiplying a term from the first sum by a term of the second sum results in a term of the following form, for some $L \subseteq \descset(\asc{v}_v)$ and $J \subseteq \descset(T \smallsetminus T_v)$:

\begin{equation} \label{eq:LJ}
    (-1)^{\size{\descset(\asc{v}_v)}-\size{L}} \cdot \prod_{i=0}^{\size{L}} \frac{1}{\prod_{u \in V(T_i)} h_u} \cdot (-1)^{\size{\descset(T \smallsetminus T_v)} - \size{J}} \cdot \prod_{i=0}^{\size{J}} \frac{1}{\prod_{u \in V(T_i)} h_u}.
\end{equation}
Note that
\[\size{\descset(\asc{v}_v)} + \size{\descset(T \backslash T_v)} = \size{\descset(\asc{v})} = \size{\descset(T)} -1 =k.\]
Moreover, the set $L \cup J$ varies over all subsets of $\descset(\asc{v})$. The tree $\asc{v}$ is obtained by connecting $\asc{v}_v$ and $T \setminus T_v$ by an edge adjacent to $v$. Therefore, the collection of subtrees obtained by subdividing $\asc{v}_v$ as prescribed by $L$ and $T \setminus T_v$ as prescribed by $J$ is the same as the one obtained by subdividing $T$ at the vertices in $R = L \cup J \cup \{v\} \subseteq \descset(T)$. So, the product~\eqref{eq:LJ} can be rewritten as

\begin{equation}\label{eq:R}
    (-1)^{k +1 - \size{R}} \prod_{i=0}^{\size{R}} \frac{1}{\prod_{v \in V(T_i)} h_v}.
\end{equation}
Using this in~\eqref{eq:exp}, we get
\begin{align*}
    \des{T}{n} =& n! \cdot \sum_{\substack{R \subseteq \descset(T) \\ v \in R}} (-1)^{k+1 - \size{R}} \prod_{i=0}^{\size{R}} \frac{1}{\prod_{u \in V(T_i)} h_u} 
    - n! \cdot \sum_{R \subseteq \descset(\asc{v})} (-1)^{k - \size{R}} \cdot  \prod_{i=0}^{\size{R}} \frac{1}{\prod_{u \in V(T_i)} h_u}\\
    =& n! \cdot \sum_{\substack{R \subseteq \descset(T) \\ v \in R}} (-1)^{k+1 - \size{R}} \cdot \prod_{i=0}^{\size{R}} \frac{1}{\prod_{u \in V(T_i)} h_u} 
    + n! \cdot \sum_{\substack{R \subseteq \descset(T) \\ v \notin R}} (-1)^{k+1 - \size{R}} \cdot \prod_{i=0}^{\size{R}} \frac{1}{\prod_{u \in V(T_i)} h_u}\\
    =& n! \cdot \sum_{R \subseteq \descset(T)} (-1)^{\size{\descset(T)} - \size{R}} \cdot \prod_{i=0}^{\size{R}} \frac{1}{\prod_{u \in V(T_i)} h_u}
\end{align*}
\noindent as desired.
\end{proof}

Our last result in this section is a formula for $\des{T}{n+1}$ in terms of $\des{T}{n}$ which will be used later in Section~\ref{sec:roots}. To see the relation, we consider which vertices of $G(T; n+1)$ can be labeled  $n+1$. Since $n+1$ is the largest label, a vertex $v$ with this label must be either the root of $G(T; n+1)$ or in $\descset$  and none of its children, if any, is in $\descset$. So, there are three cases: $v$ is the root of $G(T, n+1)$, $v$ is a descent leaf, or $v$ is a non-leaf descent with ascent children. With this in mind, we define the following subsets of $\descset$:

\[V^{\prime} =\{ v\in \descset \colon v \text{ is a leaf} \}\]
and
\[V^{\prime \prime} = \{v \in \descset \colon v \text{ is not a leaf  and the children of }v \text{ are not in } \descset \}.\]
 
 We also define the following notation. $T/v$ is the tree $T$ with the vertex $v$ deleted and the edge between the vertex $v$ and its parent contracted so that the children of $v$ become the children of the parent of $v$. In particular, if $v$ is a leaf, then $T/v = T \setminus v$. The vertices in $T/v$ that are not the original children of $v$ inherit the property of being/not being descent vertices from $T$. Furthermore, if $v$ has $c(v)$ children, then there are $2^{c(v)}$ ways to specify which of them is in $\descset$. We will denote all these $2^{c(v)}$ possibilities by $(T/v)_r$, $1 \leq r \leq 2^{c(v)}$. An example of this is shown in Figure~\ref{fig:configurations}.

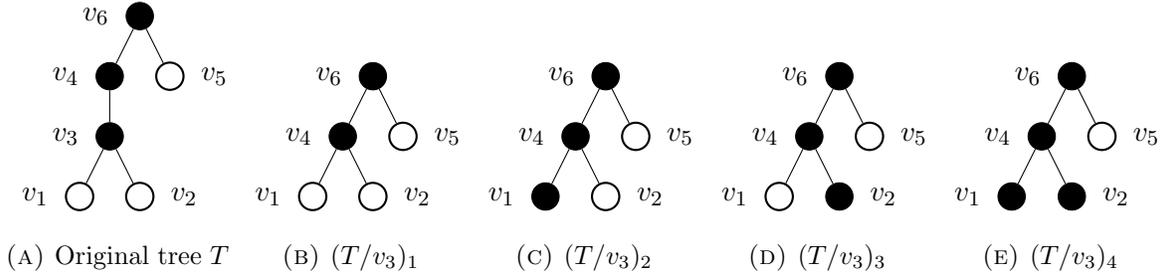
\begin{figure}[htb]
\centering
\begin{subfigure}[b]{0.18\textwidth}
\centering

\begin{tikzpicture}[level distance = .8cm, sibling distance=.8cm, nodes={draw, circle}, -,
        ascent/.style={circle, draw=black, fill=white, text=black, thick, radius=0.3},
        descent/.style={circle, draw=black, fill=black, text=white, radius=0.3,},
        ]
    \node[descent, label=left:$v_6$] {}
            child { node [descent, label=left:$v_4$]  {} 
                child { node [descent, label=left:$v_3$] {} 
                    child { node [ascent, label=left:$v_1$] {} }
                    child { node [ascent, label=right:$v_2$] {} }}}
            child { node [ascent, label=right:$v_5$] {} 
        };
\end{tikzpicture}

\caption{Original tree $T$}
\label{fig:6v-example}
\end{subfigure}
\begin{subfigure}[b]{0.18\textwidth}
\centering

\begin{tikzpicture}[level distance = .8cm, sibling distance=.8cm, nodes={draw, circle}, -,
		        ascent/.style={circle, draw=black, fill=white, text=black, thick, radius=0.3},
		        descent/.style={circle, draw=black, fill=black, radius=0.3},
		        ]
            \node[descent, label=left:$v_6$] { }
                    child { node [descent, label=left:$v_4$]{} 
                            child { node [ascent, label=left:$v_1$] { } }
                            child { node [ascent, label=right:$v_2$] { } }}
                    child { node [ascent, label=right:$v_5$] { } 
                };
\end{tikzpicture}

\caption{$(T/v_3)_1$}
\end{subfigure}
\begin{subfigure}[b]{0.18\textwidth}
\centering

\begin{tikzpicture}[level distance = .8cm, sibling distance=.8cm, nodes={draw, circle}, -,
		        ascent/.style={circle, draw=black, fill=white, text=black, thick, radius=0.3},
		        descent/.style={circle, draw=black, fill=black, text=white, radius=0.3},
		        ]
            \node[descent, label=left:$v_6$] { }
                    child { node [descent, label=left:$v_4$]{} 
                            child { node [descent, label=left:$v_1$] { } }
                            child { node [ascent, label=right:$v_2$] { } }}
                    child { node [ascent, label=right:$v_5$] { } 
                };
        \end{tikzpicture}

\caption{$(T/v_3)_2$}
\end{subfigure}
\begin{subfigure}[b]{0.18\textwidth}
\centering

\begin{tikzpicture}[level distance = .8cm, sibling distance=.8cm, nodes={draw, circle}, -,
		        ascent/.style={circle, draw=black, fill=white, text=black, thick, radius=0.3},
		        descent/.style={circle, draw=black, fill=black, text=white, radius=0.3},
		        ]
            \node[descent, label=left:$v_6$] { }
                    child { node [descent, label=left:$v_4$]{} 
                            child { node [ascent, label=left:$v_1$] { } }
                            child { node [descent, label=right:$v_2$] { } }}
                    child { node [ascent, label=right:$v_5$] { } 
                };
        \end{tikzpicture}

\caption{$(T/v_3)_3$}
\end{subfigure}
\begin{subfigure}[b]{0.18\textwidth}
\centering

\begin{tikzpicture}[level distance = .8cm, sibling distance=.8cm, nodes={draw, circle}, -,
		        ascent/.style={circle, draw=black, fill=white, text=black, thick, radius=0.3},
		        descent/.style={circle, draw=black, fill=black, text=white, radius=0.3},
		        ]
            \node[descent, label=left:$v_6$] { }
                    child { node [descent, label=left:$v_4$]{} 
                            child { node [descent, label=left:$v_1$] { } }
                            child { node [descent, label=right:$v_2$] { } }}
                    child { node [ascent, label=right:$v_5$] { } 
                };
        \end{tikzpicture}

\caption{$(T/v_3)_4$}
\end{subfigure}

\caption{$(T/v_3)_r$ for a tree $T$}
\label{fig:configurations}
\end{figure}

\begin{thm}\label{thm:d(T;n+1)}  Let $T$ be a rooted tree with descent set $\descset$, then

\begin{equation*}
\des{T}{n+1} = \des{T}{n}  + \sum_{v \in V^{\prime}} \des{T/v}{n} + \sum_{v \in V^{\prime \prime}}  \sum_{r=1}^{2^{c(v)}} \des{(T/v)_r}{n}.
\end{equation*}

\end{thm}

\begin{proof}
As we started in the discussion above, we split $D(T; n+1)$ into three subsets based on the location of the vertex $v$ labeled $n+1$: $(1)$ $v$ is the top vertex of $G(T;n+1)$, $(2)$  $v \in V ^\prime$, and $(3)$ $v \in V ^{\prime \prime}$. 

In the first case,  by deleting $v$ we get a correspondence with $\mathrm{D}(T; n+1)$. In the second case, by deleting the leaf $v$ we get a one-to-one correspondence with the labelings of $G(T \setminus v;n)  = G(T/v;n)$. In the third case, by contracting the edge between $v$ and its parent, and deleting the label $n+1$, we get a labeling of one of the trees $G((T/v)_r ; n)$. Conversely, given a labeling $w \in G((T/v)_r; n)$, one can produce a labeling of $G(T; n+1)$ by adding the label $n+1$ on $v$. In this larger labeling, all descent positions are the same as in $w$ except for the children of $v$, which are never descents. Therefore, the number of labelings of $G(T; n+1)$ in which $v \in V^{\prime \prime}$ is $\sum_{r=1}^{2^{c(v)}} \des{(T/v)_r}{n}$.

\end{proof}

\section{Expanding $\des{T}{n}$ into binomial bases} \label{S:expansions}

In this section we discuss some properties of the expansions of $\des{T}{n}$ in a couple of binomial bases. Let $\mathrm{deg}(\des{T}{n}) = h$. First consider
the following basis: 
\begin{equation*}
    a = \left\{ \binom{n-\degree}{0}, \binom{n-\degree}{1}, \dots, \binom{n-\degree}{\degree} \right\} \text{.}
\end{equation*}
When $T$ is a chain, it was shown in~\cite{diaz2019descent} that the expansion of $\des{T}{n}$  in this basis has non-negative coefficients. It was further conjectured in~\cite{diaz2019descent} and later proved by Bencs~\cite{bencs2021some} that the coefficients in such an expansion form a log-concave sequence. We prove that such results also hold for $\des{T}{n}$, when $T$ is a tree which satisfies certain conditions.

\begin{thm} \label{thm:a-positivity} Let $T$ be a tree of size $\order$ such that its root $r$ is not in $\descset$  but all of its children are in $\descset$ and let $h=s-1 = \mathrm{deg}(\des{T}{n})$. Then

\begin{equation*}
    \des{T}{n} = \sum_{k=1}^{\degree} a_k(T) \binom{n-\degree}{k}
\end{equation*}
where the coefficient $a_k(T)$ is the number of labelings $w$ in $D(T;2\degree)$ such that $\{w_1, \dots, w_{\degree}\} \cap [\degree+1, 2\degree] = [\degree+1, \degree+k]$, where $w_1, \dots, w_{\degree}$ are the labels of the vertices of $T$ in $V \setminus \{r\}$.

\end{thm}

\begin{proof}
By Corollary~\ref{cor:degree}, $\des{T}{n}$ is a polynomial of degree $\degree = \order -1$. Therefore, $\des{T}{n}$ can be written uniquely as a linear combination of the basis vectors

\begin{equation*}
    \left\{ \binom{n-\degree}{0}, \binom{n-\degree}{1}, \dots, \binom{n-\degree}{\degree} \right\} \text{.}
\end{equation*}
For a labeling $w$ of $G(T;n)$ which is in $D(T;n)$, we define

\begin{equation*}
    w[\degree] = \{w_1, \dots, w_{\degree}\} \cap [\degree+1, n],
\end{equation*}
where $w_i$ the label of vertex $v_i$ in $V \setminus \{r\}$. Further, we use this to define, for $0 \leq k \leq \degree$,

\begin{equation*}
    D_k(T; n) = \{w \in D(T;n) \colon \size{w[\degree]} = k \} \text{.}
\end{equation*}

Since $D(T;n)$ is the disjoint union of $D_k(T;n)$ for $k \geq 0$,  to prove our statement, it suffices to show $\size{D_k(T;n)} = a_k(T) \binom{n-\degree}{k}$. Note that $D_0(T;n) = \emptyset$ for the condition $\{w_1, ..., w_{\degree}\} \cap [\degree+1, n] = \emptyset$ implies that the labels of the children of $r$ are at most $\degree$, while the label of the root is at least $\degree+1$. But this contradicts the fact that the children of $r$ are descents.

Now, we assume $n \geq 2\degree$, since, if we show the equality holds for an infinite number of values, then it must hold for every value. For $k \geq 1$, consider $w \in D_k(T;n)$. There are $\binom{n-\degree}{k}$ ways of choosing the $k$ elements of $w[\degree]$. We claim that for two $k$-element sets $X, Y \subseteq [\degree+1, n]$,
\begin{equation*}
    \size{\{ w \in D_k(T;n) \colon w[h] = X \}} = \size{\{ w \in D_k(T;n) \colon w[h] = Y \}}.
\end{equation*}  
To show this, let  $f:X \to Y$ $f$ be the order preserving bijection between the sets $X$ and $Y$. The map $f$ induces a bijection $
    g: \{ w \in D_k(T;n) \colon w[h] = X \} \to \{ w \in D_k(T;n) \colon w[h] = Y \}$. For $g \in \{ w \in D_k(T;n) \colon w[h] = X \}$, we construct $u = g(w)$ by applying $f$ to the elements of $w[\degree]$ leaving the labels of $V \setminus \{r\}$ unchanged, and labeling the root $r$ of $T$ and the chain above it in ascending order using the remaining elements. This map clearly preserves the descent points everywhere except possibly at the children of the root of $T$.
    
Consider a child $c$ of $r$. First, notice $w(r) \in [\degree]$ since the vertices above the children of $r$ are labeled in increasing order and our assumption that $k \geq 1$ implies there is at least one element of $[\degree]$ not used in the labeling of $V \setminus \{r\}$. Note also, that $w(r) = u(r)$ since the map $g$ does not affect any labels in $[\degree]$. If $u(c) \in [\degree+1, n]$ then $c$ is clearly a descent in $u$; alternatively if $u(c) \in [\degree]$, then $w(c) = u(c) > u(r) = w(r)$ and $c$ is a descent in $u$.

It is not hard to see that the map $g$ is a bijection. Therefore, we have shown that for any $k$-element $X \subseteq [\degree+1, n]$,
\[\size{D_k(T;n)} = \size{\{ w \in D_k(T;n) \colon w[\degree] = X \}} \binom{n-\degree}{k}. \]One can take $X = [\degree+1, \degree+k]$. Also, note that the labels greater than $2\degree$ are  associated to the same vertices of $G(T;n)$ for all labelings in $D_k(T;n)$, independent of the value $k \leq \degree$. So,
\[a_k(T) = \size{\{ w \in D_k(T;n) \colon w[\degree] = X \}} = \size{\{ w \in D_k(T;2\degree) \colon w[\degree] = [\degree+1, \degree+k] \}}. \]
\end{proof}

\begin{lemma} \label{lem:plus-vertex-at-the-root} Let $T$ be a rooted tree and let $T^{\prime}$ be $T$ with an ascent vertex added as a parent to the root of $T$. Then,

\begin{equation*}
    \des{T}{n} = \des{T^{\prime}}{n}
\end{equation*}

\end{lemma}

\begin{proof}
This follows from the fact that $G(T; n) = G(T'; n)$ for all $n >s$, the size of $T$.
\end{proof}

\begin{cor} \label{cor:a-positivity-2}
Let $T$ be a tree of size $\order$ such that the root $r \in \descset$. Then 
\begin{equation*}
    \des{T}{n} = \sum_{k=1}^{\order} a_k(T) \binom{n-\order}{k}
\end{equation*}
where the coefficient $a_k(T)$ is the number of labelings $w$ in $D(T;2\order)$ such that $\{w_1, \dots, w_{\order} \} \cap [\order+1,2\order] = [\order+1, \order+k]$, where $w_1, \dots, w_{\order}$ are the labels of the vertices of $T$.

\end{cor}

\begin{proof}
Let $T^{\prime}$ be the tree $T$ with an ascent vertex attached on top of the root. Then the result follows by Lemma~\ref{lem:plus-vertex-at-the-root} and application of Theorem~\ref{thm:a-positivity} to $T'$.
\end{proof}


For the trees considered in Theorem~\ref{thm:a-positivity} and Corollary~\ref{cor:a-positivity-2}, the description of the coefficients $a_k(T)$ implies that they are non-negative. This property does not hold for general trees. For example, for the tree $T$ in Figure~\ref{fig:a-positivity-counter} we have
$\des{T}{n} = 560 \binom{n-7}{7} +3800 \binom{n-6}{6} + 10120  \binom{n-5}{5} + 12160 \binom{n-4}{4} + 3150 \binom{n-3}{3} - 3150 \binom{n-2}{2} - 3150 \binom{n-1}{1} - 3150 \binom{n}{0}.$ 

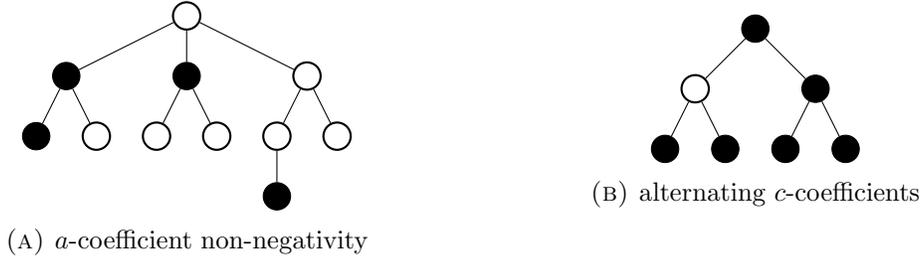
\begin{figure}[htb]
    \centering
    \begin{subfigure}[t]{0.45\textwidth}
    \centering        
    \begin{tikzpicture}[level distance = .8cm, sibling distance=.8cm, nodes={draw, circle}, -,
		        ascent/.style={circle, draw=black, fill=white, text=black, thick, radius=0.3},
		        descent/.style={circle, draw=black, fill=black, text=white, radius=0.3},
		        ]
            \node[ascent] {}
                child { node [descent] {} 
                    child { node [descent]  {} }
                    child { node [ascent] {} }
                }
                child[missing]
                child {node [descent] {} 
                    child { node [ascent] {} }
                    child { node [ascent] {} }
                }
                child[missing]
                child {node [ascent] {} 
                    child { node [ascent] {}
                        child { node [descent]{}} }
                    child { node [ascent] {} }
                };
 
        \end{tikzpicture}
    \caption{$a$-coefficient non-negativity}    
    \label{fig:a-positivity-counter}
    \end{subfigure}
    \begin{subfigure}[b]{0.45\textwidth}
    \centering
    \begin{tikzpicture}[level distance = .8cm, sibling distance=.8cm, nodes={draw, circle}, -,
		        ascent/.style={circle, draw=black, fill=white, text=black, thick, radius=0.3},
		        descent/.style={circle, draw=black, fill=black, text=white, radius=0.3},
		        ]
            \node[descent] {}
                    child { node [ascent]  {} 
                        child { node [descent] {} }
                        child { node [descent] {} }}
                    child [missing]
                    child { node [descent]  {} 
                        child { node [descent] {} }
                        child { node [descent] {} }
                };
        \end{tikzpicture}
        \caption{alternating $c$-coefficients}
        \label{fig:c-basis-counter}
        \end{subfigure}
    \caption{Trees for counterexamples}    

\end{figure}


 However, generalizing the result about permutations, we prove that the sequence $a_k(T)$ is log-concave in the cases covered by Theorem~\ref{thm:a-positivity} and Corollary~\ref{cor:a-positivity-2}. We follow the approach in~\cite{bencs2021some} and temporarily shift our focus to expansions in another binomial basis. Namely, the set
\[\bar{a} = \left\{\binom{n-\degree-1}{0} ,  \binom{n-\degree}{1}, \dots, \binom{n-1}{\degree} \right\}\] is another basis for the polynomials of degree up to $\degree$ and, therefore, $\des{T}{n}$  can be uniquely expanded as

\begin{equation*}
    \des{T}{n} = \abar{-1} \binom{n-\degree-1}{0} + \abar{0} \binom{n-\degree}{1} + \cdots + \abar{\degree-1} \binom{n-1}{\degree} \text{.}
\end{equation*}
Although any polynomial can be written in this basis, we restrict our following discussion to the classes of trees for which we've proven the nonnegativity of $a_k(T)$. For the remainder of this section $T$ is a tree in which  root is not a descent vertex but all of its children are. Note that
\begin{equation*}
    \des{T}{\degree} = \abar{-1} \binom{-1}{0} + \sum_{k=1}^{\degree} \binom{k}{k+1} = \abar{-1}.
\end{equation*} As we prove in Section~\ref{sec:roots}, where we discuss the roots of $\des{T}{n}$, for these trees we have $\des{T}{\degree}=0$ and, therefore $\abar{-1}=0$. We will show that
the sequence $\{\abar{k}\}_{k=0}^{\degree}$ counts certain linear extensions of a poset $P_T$ associated to the tree $T$. Before we define $P_T$, we introduce some poset terminology we will need.

Let $P$ be a finite poset and $v \in P$ a fixed element. Let $Ext(P)$ be the set of order-preserving bijections from $P$ to the chain $[1, 2, \dots, \size{P}]$. The height polynomial of $v$ in $P$ is defined as

\begin{equation*}
    h_{P,v}(x) = \sum_{\phi \in Ext(P)} x^{\phi(v)-1} = \sum_{k=0}^{\size{P}-1} h_k(P,v) x^k.
\end{equation*}
So, $h_k(P,v)$ is the number of linear extensions  of $P$ in which $v$ is labeled $k+1$.

If the root of $T$ is not in $\descset$, but all of its children are, then $\mathrm{deg}(\des{T}{n}) = h = \size{V}-1$.  Let  $V = \{v_1, \dots, v_{h+1}\}$. Then $P_T$ is a poset on $\{u_1, \dots, u_{\degree+1}\}$ defined in the following way.  If $v_i$ is the child of $v_j$, then $u_i > u_j$ if $u_i \in \descset$ and $u_i < u_j$ if $u_i \notin \descset$.  An example of this construction is shown in Figure~\ref{fig:tree-to-poset}. The connection between the poset $P_T$ and the coefficients $\bar{a}_k$ is given in the folowing proposition.

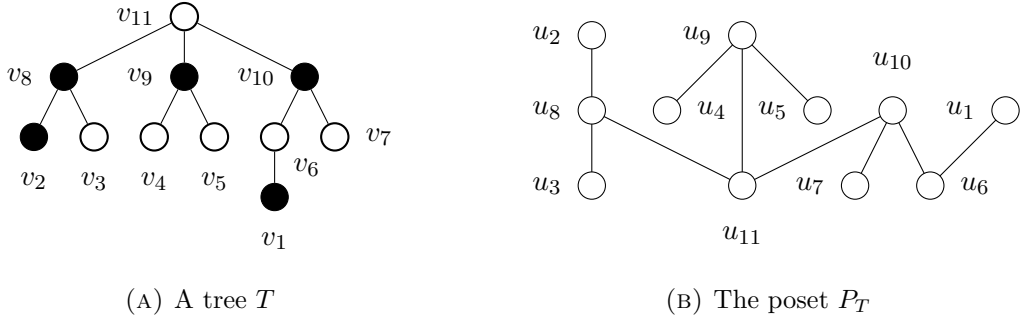
\begin{figure}[t]
    \centering
    \begin{subfigure}[b]{0.45\textwidth}
    \centering        \begin{tikzpicture}[level distance = .8cm, sibling distance=.8cm, nodes={draw, circle}, -,
		        ascent/.style={circle, draw=black, fill=white, text=black, thick, radius=0.3},
		        descent/.style={circle, draw=black, fill=black, text=white, radius=0.3},
		        ]
            \node[ascent, label=left:$v_{11}$] {}
                child { node [descent, label=left:$v_8$] {} 
                    child { node [descent, label=below:$v_2$]  {} }
                    child { node [ascent, label=below:$v_3$] {} }
                }
                child[missing]
                child {node [descent, label=left:$v_9$] {} 
                    child { node [ascent, label=below:$v_4$] {} }
                    child { node [ascent, label=below:$v_5$] {} }
                }
                child[missing]
                child {node [descent, label=left:$v_{10}$] {} 
                    child { node [ascent, label=320:$v_6$] {}
                        child { node [descent, label=below:$v_1$]{}} }
                    child { node [ascent, label=right:$v_7$] {} }
                };
 
        \end{tikzpicture}
    \caption{A tree $T$}    
    \end{subfigure}
    \begin{subfigure}[b]{0.45\textwidth}
    \centering
    \begin{tikzpicture}[scale=0.5, nodes={draw=black, fill=white, circle}]
        \draw (0,0) -- (0,4) -- (-2,2); 
        \draw (0,4) -- (2,2); 
        \draw (0,0) -- (-4,2) -- (-4,4); 
        \draw (-4,2) -- (-4,0); 
        \draw (0,0) -- (4,2) -- (5,0) -- (7,2); 
        \draw (4,2) -- (3,0); 
        
        \node[label=below:$u_{11}$] at (0,0) {};
        \node[label=left:$u_9$] at (0,4) {};
        \node[label=right:$u_4$] at (-2,2) {};
        \node[label=left:$u_5$] at (2,2) {};
        \node[label=above:$u_{10}$] at (4,2) {};
        \node[label=left:$u_7$] at (3,0) {};
        \node[label=right:$u_6$] at (5,0) {};
        \node[label=left:$u_1$] at (7,2) {};
        \node[label=left:$u_8$] at (-4,2) {};
        \node[label=left:$u_2$] at (-4,4) {};
        \node[label=left:$u_3$] at (-4,0) {};
        
    \end{tikzpicture}
    
    \caption{The poset $P_T$}
    \label{fig:poset}
    \end{subfigure}
    
    \caption{The poset $P_T$ associated to a tree $T$ with a marked descent set $\descset$}
    \label{fig:tree-to-poset}
\end{figure}

\begin{prop} \label{prop:a-bar-positivity} Let $T$ be a tree with descent set $\descset$ such that the root of $T$ is not in $\descset$, but all of its children are. Then for $0 \leq k \leq \degree-1$
\begin{equation*}
    \abar{k+1} = h_{\degree-k}(P_T,u_{\degree+1}) \text{.}
\end{equation*}
\end{prop}

\begin{proof}
    Since $\des{T}{n}$ is a polynomial of degree $\degree$, we can write it uniquely as a linear combination of the basis vectors
    \begin{equation*}
        \left\{\binom{n-\degree-1}{0}, \binom{n-\degree}{1}, \dots, \binom{n-1}{\degree}\right\} \text{.}
    \end{equation*}
    So, it suffices to show that for $n \geq \degree$ we have
    \begin{equation*}
        \des{T}{n} = \sum_{k=0}^{\degree-1} h_{\degree-k}(P_T, u_{\degree+1}) \binom{n-\degree+k-1}{k} \text{.}
    \end{equation*}
    
   Let $v_{\degree + 1}$ be the root of $T$. Let $B_k(T;n) = \{w \in D(T;n) \colon w(v_{\degree + 1})=k+1\}$ for $0 \leq k \leq n-1$. By the assumption on $T$, the children of  $v_{\degree+1}$ are descents and above $v_{\degree+1}$ there is an ascending chain in $G(T; n)$. There are $n-\degree-1$ elements in the ascending chain and at least one child of $v_{\degree+1}$  must have a label higher  than $v_{\degree+1}$. Thus, $w(v_{\degree+1}) \leq \degree$ and $B_k(T;n) = \emptyset$ for $\degree \leq k \leq n-1$. So, $D(T,n)$ is a disjoint union of the sets $B_k(T;n)$ for $0 \leq k < \degree$. Additionally, note $\size{B_k(T;\degree+1)} = h_k(P_T; u_{\degree+1})$. 
    
   We claim
    \begin{equation*}
        \size{B_k(T;n)} = \size{B_k(T;\degree+1) \times \binom{[k+2, n]}{\degree-k}} = \size{B_k(T;\degree+1)}\binom{n-k-1}{\degree-k}.
    \end{equation*}

    To prove the first equality, we establish a bijection. If $w \in B_k(T;n)$, then let $V_w = \{w (v_i) \colon 1 \leq i \leq \degree \text{ and } w(v_i) > k+1\}$. Clearly, $V_w \subseteq [k+2,n]$ and  $\size{V_w} = \degree-k$. Let $w'$ be the standardization of the restriction of $w$ on the tree $T$. In other words, $w'$ uses the labels $\{1, \dots, h+1\}$ and has the property that for $v_1, v_2 \in V$, $w'(v_1) < w'(v_2)$ if and only if $w(v_1) < w(v_2)$. Then $w' \in  B_k(T; \degree+1)$.  
    
 Let $f: B_k(T;n) \to B_k(T; \degree+1) \times \binom{[k+2,n]}{\degree-k}$ be defined by
    \begin{equation*}
        f(w) = (w', V_w). 
    \end{equation*}
   Checking whether $f$ is a bijection is rather simple and left to the reader.
    
    Then, we have 
    \begin{equation*}
        \des{T}{n} = \size{D(T;n)} = \size{\bigcup_{k=0}^{\degree-1}B_k(T;n)} = \sum_{k=0}^{\degree-1} \size{B_k(T;n)} = 
    \end{equation*}
    \begin{align*}
        &= \sum_{k=0}^{\degree-1} \size{B_k(T;\degree+1) \times \binom{[k+2, n]}{\degree-k}}  \\ &=
        \sum_{k=0}^{\degree-1} \size{B_k(T;\degree+1} \binom{n-k-1}{\degree-k}  \\ &=
        \sum_{k=0}^{\degree-1} h_k(P_T, u_{\degree+1}) \binom{n-k-1}{\degree-k} \\ &=
        \sum_{l=1}^{\degree} h_{\degree-l}(P_T, u_{\degree+1})  \binom{n-\degree+l -1}{l}.
    \end{align*}   
\end{proof}

The following properties of the sequence $\{h_k(P,v)\}_{k=1}^{\size{P}}$ follow from two results of Stanley~\cite{stanley1981two, stanley1986two}.
\begin{thm}\cite[Theorem 2.3]{bencs2021some}\label{thm:bencs-posets}
Let $P$ be a finite poset, and $v \in P$ be fixed. Then the coefficient sequence $\{h_k(P,v)\}_{k=1}^{\size{P}}$ is log concave. Moreover if all comparable elements with $v$ are bigger than $v$ in $P$, then $\{h_k(P,v)\}_{k=1}^{\size{P}}$ is a decreasing, log-concave sequence. 
\end{thm}

Combining Theorem~\ref{thm:bencs-posets} and Proposition~\ref{prop:a-bar-positivity}, we directly get the following corollary.

\begin{cor} \label{cor:abar}
   Let $T$ be a tree with descent set $\descset$ such that either its root is a descent or all the children of its root are descents, then the sequence $\abar{0}, \abar{1}, \dots, \abar{\degree-1}$ is an increasing, log-concave sequence of nonnegative integers.
\end{cor}

We are now ready to go back to the sequence $a_0(T), a_1(T), \dots, a_{\degree}(T)$.

\begin{thm} \label{a-log-concavity}
    Let $T$ be a tree with descent set $\descset$ such that either its root is a descent or all the children of its root are descents, then the sequence $a_0(T), a_1(T), \dots, a_{\degree}(T)$ is a log-concave sequence of nonnegative integers.
\end{thm}

\begin{proof}  Let \begin{align*}
    a(T,x) &= \sum_{k=0}^ha_k(T)x^k\\
    \bar{a}(T,x) &= \sum_{k=0}^{h-1}\abar{k}x^k.\\
\end{align*}  Proposition 3.3 of~\cite{bencs2021some} states that \begin{equation} \label{rel} a(T,x) = x\bar{a}(T,x+1)
\end{equation}
for a nonempty tree $T$ which is chain but the proof of this fact does not depend on the tree $T$. It only depends on the fact that the coefficients of $a(T,x)$ and $\bar{a}(T,x)$ come from expansions of the same polynomial $\des{T}{n}$.  By Corollary~\ref{cor:abar} we know that the coefficient sequence of the polynomial $a(T, x)$ is log-concave, and consequently
has no internal zeros. It is known that this implies that the coefficient sequence of the polynomial $a(T, x + 1)$ is also log-concave~\cite{brenti1989unimodal}. Since multiplication
with an $x$ only shifts the coefficient sequence, the coefficient sequence of $xa(T, x + 1) = a(T, x)$ is also log-concave.
\end{proof}

In~\cite{diaz2019descent}, the expansion $d(T;n)$ into the basis

\begin{equation*}
    c = \left\{ \binom{n+1}{0}, \binom{n+1}{1}, \dots, \binom{n+1}{\degree} \right\}
\end{equation*}
when $T$ is a chain is also considered. Let \[ \des{T}{n} = \sum_{k=0}^h c_k(T) \binom{n+1}{k}.\] It was conjectured in~\cite{diaz2019descent} and later proved in~\cite{bencs2021some}, that for a chain $T$, the coefficients $c_0(T),  \dots, c_h(T)$ are integers alternating in sign. 

That the coefficients are integers for general trees $T$ can be deduced using induction on the degree of $\des{T}{n}$ and the recursion in Theorem~\ref{thm:d(T;n+1)} . Namely, we have
\begin{equation}\label{eq:cint}
\des{T}{n+1} - \des{T}{n}  = \sum_{v \in V^{\prime}} \des{T/v}{n} + \sum_{v \in V^{\prime \prime}}  \sum_{r=1}^{2^{c(v)}} \des{(T/v)_r}{n}.
\end{equation} The left-hand side of~\eqref{eq:cint} is 
\[\des{T}{n+1} - \des{T}{n} = \sum_{k=0}^h c_k(T) \left(\binom{n+2}{k} - \binom{n+1}{k}\right) = \sum_{k=0}^h c_k(T) \binom{n+1}{k+1}. \] The descent polynomials on the right-hand side of~\eqref{eq:cint} are all of degree less than $h$, and therefore their coefficients in the expansion in the basis $c$ are integers, by the induction hypothesis.

However, the alternating sign property of $c_0(T),  \dots, c_h(T)$ does not extend to general trees, and not even to the tree classes for which we have shown that the $a$-coefficients are nonnegative and log-concave. We show a counterexample. Take the tree shown in Figure~\ref{fig:c-basis-counter}; its descent polynomial is $\des{T}{n} = 60 \binom{n-1}{7} - 60 \binom{n-1}{6} + 20 \binom{n-1}{5} + 44 \binom{n-1}{4} - 120 \binom{n-1}{3} + 200 \binom{n-1}{2} - 280 \binom{n-1}{1} + 360 \binom{n-1}{0}$ which does not have alternating coefficients.


\section{The roots of $\des{T}{n}$} \label{sec:roots}

It was conjectured in~\cite{diaz2019descent}, and proved in~\cite{jiradilok2019roots} and \cite{bencs2021some} that when $T$ is a chain, the degree of $\des{T}{n}$ is a bound on the roots of the polynomial, i.e, if $z \in \mathbb{C}$ is a root of $\des{T}{n}$, then $\size{z} \leq \degree$. In addition, that $\mathcal{R}(z) \geq -1$ was conjectured in~\cite{diaz2019descent} and proved in~\cite{jiradilok2019roots}. 

It is natural to ask if the bounds extend to  the roots of $\des{T}{n}$ for general trees $T$.  The answer in general is, no. For example, let $T$ be the tree in Figure~\ref{fig:bound-counter}. Its descent polynomial $\des{T}{n} = \frac{1}{3}x^3 - x^2 - \frac{58}{3}x + 80$ has roots $-8$, $5$, and $6$.

\begin{figure}[htb]
    \centering
    \begin{tikzpicture}[level distance = .8cm, sibling distance=.8cm, nodes={draw, circle}, -,
		        ascent/.style={circle, draw=black, fill=white, text=black, thick, radius=0.3},
		        descent/.style={circle, draw=black, fill=black, text=white, radius=0.3},
		        ]
            \node[ascent] {}
                    child { node [ascent]  {} 
                        child { node [ascent] {} }
                        child { node [ascent] {} }}
                    child [missing]
                    child { node [descent]  {} 
                        child { node [descent] {} }
                        child { node [ascent] {} }
                };

        \end{tikzpicture}
    
    \caption{A tree $T$ for which the roots of $\des{T}{n}$ exceed $h$}
    \label{fig:bound-counter}
\end{figure}
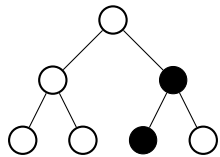

\begin{prop} \label{prop:bound}
    Let $T$ be a tree with descent set $V_D$ such that either the root of $T$ is a descent or all the children of the root are descents. Let $z \in \mathbb{C}$ such that $\des{T}{z} = 0$, then $\size{z} \leq \degree$.
\end{prop}
\begin{proof} Similarly as in Corollary 5.3 of~\cite{bencs2021some}, one can consider the polynomial $p(z) =(z-1)\bar{a}(T, z)$. Since, by Corollary~\ref{cor:abar}, the coefficients of $\bar{a}(T, z)$ form an increasing sequence, one can readily see that the coefficients $p_k$, except for $p_h$, of $p(z) = \sum_{k=0}^h p_kz^k$, are all non-positive. Also, their sum is 0 and
\[ \sum_{k=0}^{h-1} |p_k| = -
\sum_{k=0}^{h-1}
p_k = p_{h}  > 0.\] Therefore, by Lemma 5.2 from~\cite{bencs2021some}, if $|z| > h$, then
\[d(T, z) = \sum_{k=0}^h p_k \binom{z-m+k}{k} \neq 0.\]
\end{proof}

The integer roots of $\des{T}{n}$ are bounded for general trees as can be seen from the following result.

\begin{prop}
    If $z \in \mathbb{Z}$ and $\des{T}{z} = 0$, then $z \leq \order$ where $\order$ is the size of $T$. 
\end{prop}

\begin{proof}
    As discussed in Section~\ref{S:introduction}, for any integer $z \geq \order$, there is at least one labeling of $G(T;z)$ in $D(T;z)$ which forces $\des{T}{z} > 0$. Thus, any integer root of $\des{T}{n}$ must be less than or equal to $\order$.
\end{proof}

The remainder of this section is devoted to results about when certain integers are roots of $\des{T}{n}$. Recall that $\natlab{T}$ denotes the number of natural labelings of $T$.


\begin{lemma} \label{lem: des of zero}
    For any tree $T$, $\des{T}{0}= (-1)^{|\descset|} \cdot \natlab{T}$.
\end{lemma}
\begin{proof}
To see why this is true, consider the  recursion from Proposition~\ref{prop:recursion} fully expanded. The only term in the expansion that is constant in $n$ is equal to $(-1)^{|\descset|} \cdot \natlab{T}$. 
\end{proof}


\begin{lemma} \label{lem: full roots}
    For a tree $T$ of size $\order$ with $\descset = V$, the roots of the polynomial $\des{T}{n}$ are $1, 2, 3, \ldots, \order$.
\end{lemma}

\begin{proof}
    We attempt to label such a tree with $[n]$. The label $1$ cannot be placed at or below the root, because all of these nodes must be larger than their parent. Therefore, it must be placed in the chain of ascents. Since the chain is arranged in ascending order, the label $1$ must be placed at the parent of the root. We can select the $\order$ labels to be placed below the 1 from the remaining $n-1$ labels. The leftovers are fixed to the vertices above the 1 in ascending order. For any selection of $\order$ vertices there is some constant number of ways to arrange them in the tree such that they are all at descent points. This constant $c$ does not depend on which $\order$ labels were selected. Therefore, 
    $$ \des{T}{n} = \binom{n-1}{\order}\cdot c$$
    The roots of this polynomial are $1, 2, 3, \ldots, \order$, as claimed.
\end{proof}


\begin{lemma} \label{lem:size root}
    For a tree $T$ of size $\order$, $\order$ is a root of $\des{T}{n}$ if and only if the root of $T$  is in $\descset$.
\end{lemma}

\begin{proof}
    This follows from the discussion in Section~\ref{S:introduction}.
\end{proof}


\begin{thm} \label{thm:s-1 root}
    For a tree $T$ of size $\order$, $\order - 1$ is a root of $\des{T}{n}$ if and only if the root of $T$ has at least one child in $\descset$.
\end{thm}

\begin{proof}
    Let $T$ be a tree where the root is in $\descset$ and has at least one child in $\descset$. Let $T'$ be identical to $T$ but with an ascent root. Consider applying the recursion from Proposition~\ref{prop:recursion} where $v$ is the root of the tree and $n=\order-1$:
    \begin{align*}
        \des{T}{\order-1} &= \binom{\order-1}{\order} \cdot \des{\asc{v}_v}{h_v} \cdot \des{T \setminus T_v}{\order-1-h_v} - \des{\asc{v}}{\order-1} \\
            &= - \des{\asc{v}}{\order-1}
    \end{align*}
    From this, we see that $\order-1$ is a root of $\des{T'}{n}$ if and only if it is a root of $\des{T}{n}$. So, for the rest of this proof it is sufficient to consider only a tree $T$ with its root in $\descset$.
    
    For such a tree $T$ we consider the recursion from Theorem~\ref{thm:d(T;n+1)} evaluated at $n = \order-1$:  
    \begin{equation*}
        \des{T}{\order} = \des{T}{\order-1} + \sum_{v \in V^{\prime}} \des{T/v}{\order-1} + \sum_{v \in V^{\prime \prime}} \left( \sum_{r=1}^{2^{c(v)}} \des{(T/v)_r}{\order-1} \right) .
    \end{equation*}
    Since $T$ has a descent root, $\des{T}{\order}$ = 0 by Lemma~\ref{lem:size root}. Note that the root is not in $V'$, which means that the root is never removed in either summation. Therefore, every tree $T/v$ in either summation has size $s-1$ and a descent root. By Lemma~\ref{lem:size root} once more, every term in both summations is therefore equal to 0. We are left with 
    \begin{equation*}
        0 = \des{T}{\order-1},
    \end{equation*}
    which is what we wanted to show.
    
    We will prove the reverse statement through induction on $\order$. For any tree $T$ with $\order = 1$, the root of the tree has no descent children and $\order-1$ is not a root of $\des{T}{n}$. Now we assume that $\order-1$ is not a root of $\des{T}{n}$ for all trees with $\order < k$ that have roots with no children in $\descset$.
    
    Now let $T$ be a tree of size $k$ whose root is in $\descset$ but has no children in $\descset$. We apply the recursion from Corollary~\ref{cor:ascent recursion} where $v$ is a child of the root and $n = k-1$:
    \begin{equation*}
        \des{T}{k-1} = \binom{k-1}{h_v} \cdot \des{T_v}{h_v} \cdot \des{T \setminus T_v}{k-1-h_v} - \des{\desc{v}}{k-1}
    \end{equation*}
    From the forwards direction of this proof, we know that $\des{\desc{v}}{k-1} = 0$. We know by construction that $h_v \leq k-1$, so $\binom{k-1}{h_v}$ does not equal 0. By Lemma~\ref{lem:size root}, we know that $\des{T_v}{h_v}$ does not equal 0. By construction, $T\setminus T_v$ is a tree with $\order<k$ whose root has no descent children and $k-1-h_v$ is equal to the size of $T\setminus T_v - 1$, which means that by our inductive hypothesis, $\des{T \setminus T_v}{k-1-h_v}$ does not equal 0. Therefore, $\des{T}{k-1}$ does not equal 0, which is what we wanted to show. 
\end{proof}


\begin{thm}
    For a tree $T$ with size $\order$ and whose root has $k$ children that are in $\descset$ (where $k>0$), $\order-1, \order-2, \ldots, \order-k$ are roots of $\des{T}{n}$.
\end{thm}
\begin{proof}
    Let $T$ be a tree whose root has $k>0$ children that are in $\descset$. We know from Theorem~\ref{thm:s-1 root} that $\order-1$ is a root of $\des{T}{n}$ for all trees with at least 1 child of the root in $\descset$. Now we'll assume that $\order-m$ is a root of $\des{T}{n}$ for all trees with at least $m$ children of the root in $\descset$ and show that $\order-(m+1)$ is a root of $\des{T}{n}$ for $m+1 \leq k$. We consider the recursion from Theorem~\ref{thm:d(T;n+1)} evaluated at $n = \order-(m+1)$:
    \begin{equation*}
        \des{T}{\order-m} = \des{T}{\order-(m+1)} + \sum_{v \in V^{\prime}} \des{T/v}{\order-(m+1)} + \sum_{v \in V^{\prime \prime}} \left( \sum_{r=1}^{2^{c(v)}} \des{(T/v)_r}{\order-(m+1)} \right) .
    \end{equation*}
    By our inductive hypothesis, the left-hand side of this equation is 0. Now, note that for every tree $T'$ in the summations, the root of $T'$ has at least $m$ descent children and the size of the tree is $\order-1$. So, by our inductive hypothesis, all of the summation terms also evaluate to 0 (to see why this is true, think of $\order-(m+1)$ as $(\order-1)-m$). These observations leave us with 
    \begin{equation*}
        0 = \des{T}{\order-(m+1)}
    \end{equation*}
    which is what we wanted to show. 
\end{proof}

    Note that this is not a complete characterization of these roots. There are many examples of trees $T$ where $\des{T}{n}$ has $\order-m$ as a root while the root of the tree has fewer than $m$ children in $\descset$. For example, the tree in Figure~\ref{fig:6v-example} has the polynomial 
    \begin{equation*}
        \des{T}{n} = \frac{x^6}{36} - \frac{5x^5}{12} + \frac{19x^4}{9} - \frac{49x^3}{12} + \frac{103x^2}{36} - \frac{x}{2} - 10,
    \end{equation*} 
    which has a root of 4 even though the root of the tree has less than two children in $\descset$.
 
\begin{thm} \label{thm: one is root}
    For a tree $T$, 1 is a root of $\des{T}{n}$ if and only if all the leaves in $T$ are in $\descset$. 
\end{thm}

\begin{proof}
    Let $T$ be a tree with arbitrary descent structure. We consider the equation from Proposition~\ref{prop:recursion} evaluated at $n=1$, where $v$ is not a leaf:
    $$ \des{T}{1} = \binom{1}{h_v} \cdot \des{\asc{v}_v}{h_v} \cdot \des{T \smallsetminus T_v}{1-h_v} - \des{\asc{v}}{1}.$$
    Since $v$ is not a leaf, $h_v > 1$. Therefore, $\binom{1}{h_v}$= 0 and the equation simplifies to
    $$ \des{T}{1} = - \des{\asc{v}}{1}.$$
    Thus, for any tree, changing a non-leaf vertex from a descent point to an ascent point swaps the sign of $\des{T}{1}$ without affecting its magnitude.
    
    First, we consider the forwards direction. Let $T$ be a tree with descent points at all of its leaves. Let $T'$ be a tree with the same vertex structure but where all the vertices are descents. Note that we can transform $T'$ into $T$ by turning the appropriate non-leaf descent points into ascent points. Therefore, $\des{T}{1}$ and $\des{T'}{1}$ have the same magnitude. From Lemma~\ref{lem: full roots} we know that $\des{T'}{1}=0$. Therefore, $\des{T}{1} = 0$, as desired.
    
    Now, we consider the backwards direction. Let $T$ be a tree with $k$ non-descent leaves. Consider applying the recursion from Corollary~\ref{cor:ascent recursion} at one such leaf. This produces the following equation:
    \begin{equation*}
        \des{T}{1} = \des{T_1}{0} - \des{T_1'}{1}
    \end{equation*}
    where $T_n$ is a tree identical to $T$ but with 1 leaf ascent removed and $n-1$ leaf ascents turned to descents and $T_n'$ is identical to $T_n$ except the first leaf is turned into a descent rather than removed.
    
    We continue to apply the recursion on the right-hand term until we end up with the following: 
    \begin{equation*}
        \des{T}{1} = \des{T_1}{0} - [\des{T_2}{0} - [\des{T_3}{0} - \cdots - [\des{T_k}{0} - \des{T_k'}{1}]]].
    \end{equation*}
    Taking into account the nested subtractions, this becomes
    \begin{equation*}
        \des{T}{1} = (-1)^k\cdot\des{T_k'}{1} + \sum_{n=1}^k (-1)^{n-1} \cdot \des{T_n}{0}
    \end{equation*}
    By construction, all of the leaves in $T_k'$ are descents, which means that we can apply the forwards direction of this proof and reduce $\des{T_k'}{1}$ to 0. 
    
    Consider now that $T_n$ has one fewer descent than $T_{n+1}$, so by Lemma~\ref{lem: des of zero}, $\des{T_n}{0}$ and $\des{T_{n+1}}{0}$ are non-zero numbers with opposite signs. In fact, we know that the sign of $\des{T_1}{0}$ is $(-1)^{|\descset|}$, where $\descset$ is the set of descents in the original tree $T$. Taking these three observations into account, the simplified expression for $\des{T}{1}$ is
    \begin{align*}
        \des{T}{1} &= \sum_{n=1}^k (-1)^{n-1} \cdot (-1)^{|\descset| -n} \cdot(-1) \cdot \natlab{T_n} \\
        &= (-1)^{|\descset|} \sum_{n=1}^k\natlab{T_n}
    \end{align*}
    where $\natlab{T_n}$ is the number of natural labelings of $T_n$. Since the summation is a series of strictly positive numbers, $\des{T}{1}$ does not equal 0, which is what we wanted to show. 
\end{proof}

\begin{prop} \label{prop:same-roots}
    Let $T$ be a tree of size $\order$ such that its only descent is at the root, then \[\des{T}{n} = \left( \binom{n}{\order} - 1 \right) \cdot \natlab{T}\].
    
    In particular, if $T$ and $T^{\prime}$ are two trees of size $\order$ such that their only descent is their roots, then their descent polynomials have the same roots. 
\end{prop}

\begin{proof}
    To show this, we calculate $\des{T}{n}$ for some $T$ as defined in our proposition using the recursion found in Proposition~\ref{prop:recursion}. For $r$ the root of $T$, we find,
    \begin{equation*}
        \des{T}{n} = \binom{n}{\order} \des{\asc{\text{r}}_{\text{r}}}{\order} \des{T \smallsetminus T_{\text{r}}}{n - \order} - \des{\asc{\text{r}}}{n} \text{.}
    \end{equation*}
    
    However $\des{T \smallsetminus T_{\text{r}}}{n - \order} = 1$ since it labels only a path with no descents. Note, also that $\natlab{T} = \natlab{G(T;n)}$. We quickly prove this by calculating $G(T;n+1)$ for $n \geq \order$. Using the formula for $\natlab{T}$ from Section~\ref{S:introduction} we can write \[\natlab{G(T;n+1)} = \frac{(n+1)!}{\prod_{v \in V(G(T;n+1))}h_v} = \frac{n! \cdot (n+1)}{\prod_{v \in V(G(T;n))}h_v \cdot (n+1)} = \natlab{G(T;n)}.\] Since $r$ is the only descent, this implies $\des{\asc{\text{r}}_{\text{r}}}{\order} = \des{\asc{\text{r}}}{n} = \natlab{T}$. So, we get,
    \begin{equation*}
        \des{T}{n} = \left( \binom{n}{\order} - 1 \right) \cdot \natlab{T}
    \end{equation*}
    as desired.
\end{proof}

\begin{cor} \label{cor:-1-root}
    Let $T$ be a tree of size $\order$ such that $T$ only has a descent at its root. Then, $\des{T}{-1} = 0$ if and only if $\order$ is even.
\end{cor}

\begin{proof}
    Take the formula for $\des{T}{n}$ we found in Proposition~\ref{prop:same-roots}, and plug in $-1$, we get
    \begin{align*}
        \des{T}{-1} &= \left( \binom{-1}{\order} - 1 \right) \cdot \natlab{T} \\
        &= \left( \frac{(-1)(-2) \cdots (-\order)}{\order!} - 1 \right) \cdot \natlab{T} \\
        &= ((-1)^{\order} - 1)\cdot \natlab{T} \text{.}
    \end{align*}
    
    Here, if $\order$ is even, then $(-1)^{\order} - 1 = 0$ and $-1$ is a root. But, if $\order$ is odd, $(-1)^{\order} - 1 = -2$ and $-1$ is not a root.
\end{proof}


\end{document}